\documentclass[11pt, letterpaper, oneside]{article}

\usepackage{latexsym, amsmath, amssymb, amsfonts, amscd}
\usepackage{amsthm}
\usepackage{t1enc}
\usepackage[mathscr]{eucal}
\usepackage{indentfirst}
\usepackage{graphicx}   
\usepackage{fancyhdr}
\usepackage{enumerate}
\usepackage[all,poly,web,knot]{xy}

\theoremstyle{plain}
\newtheorem{thm}{Theorem}[section]

\newtheorem{prop}[thm]{Proposition}
\newtheorem{lemma}[thm]{Lemma}
\newtheorem{cor}[thm]{Corollary}

\renewcommand{\latticebody}{\drop@{ }}

\theoremstyle{definition}
\newtheorem{defi}[thm]{Definition}

\newtheorem{notn}[thm]{Notation}

\theoremstyle{remark}
\newtheorem{remark}[thm]{Remark}

\newcommand{\lra}{\longrightarrow}

\newcommand{\rra}{\Rightarrow}

\newcommand{\N}{\ensuremath{\mathbb N}}
\newcommand{\Z}{\ensuremath{\mathbb Z}}

\newcommand{\R}{\ensuremath{\mathbb R}}

\newcommand{\g}{\ensuremath{\mathfrak{g}}}

\newcommand{\cF}{\mathcal{F}}

\newcommand{\cX}{\mathcal{X}}
\newcommand{\cY}{\mathcal{Y}}

\newcommand{\cG}{\mathcal{G}}
\newcommand{\cH}{\mathcal{H}}

\DeclareMathOperator{\pr}{pr} 
\DeclareMathOperator{\id}{id}

\DeclareMathOperator{\Aut}{Aut}

\newcommand{\ta}{\tilde{a}}

\newcommand{\tf}{\tilde{f}}

\newcommand{\tx}{\tilde{x}}

\newcommand{\bt}{\mathbf{t}}                  
\newcommand{\bs}{\mathbf{s}}                  

\def\PB(#1,#2,#3,#4){
\left\{\begin{matrix}#1&\!\!\!\stackrel{?}{\longrightarrow}&\!\!\!#2\\
\downarrow&&\!\!\!\downarrow\\
#3&\!\!\!\stackrel{?}{\longrightarrow}&\!\!\!#4\end{matrix}\right\}}

\def\pb(#1,#2,#3,#4){ \hom(#1 \to #3, #2 \to #4)}

\newcommand\xto{\xrightarrow}
\newcommand\transpose[1]{{^{\mathit t} #1}}

\newcommand\TU{{\mathbf 1}}

\newcommand\HiSk{\underline{\mathfrak{HS}}}
\newcommand\DSta{\underline{\smash[b]{\mathfrak{DSta}}}}
\newcommand\LGpd{\underline{\smash[b]{\mathfrak{LGpd}}}}

\newcommand\Auto{\mathrm{Aut}}
\newcommand\Str{\mathsf{Strict}}
\newcommand\Class{\mathsf{Stack}}
\newcommand{\pt}{\star}

\begin{document}

\title{Strictification of \'etale stacky Lie groups\thanks{Supported by Deutsche Forschungs\-gemein\-schaft (DFG; German Research Foundation) through the Institutional Strategy of the University of G\"ottingen}}
\author{Giorgio Trentinaglia and Chenchang Zhu}
\date{\itshape\footnotesize Courant Research Centre ``Higher Order Structures''\\University of G\"ottingen, Germany}
\maketitle

\begin{abstract}
\noindent We define stacky Lie groups to be group objects in the 2-category of differentiable stacks. We show that every connected and \'etale stacky Lie group is equivalent to a crossed module of the form $(\Gamma, G)$ where $\Gamma$ is the fundamental group of the given stacky Lie group and $G$ is the connected and simply connected Lie group integrating the Lie algebra of the stacky group. Our result is closely related to a strictification result of Baez and Lauda.
\end{abstract}

\section{Introduction}

Over the last few years there has been a lot of interest in the so-called {\em higher groups} \cite {Baez:2gp, henriques, blohmann}. As the name suggests, a higher group should be regarded as a ``generalized group'' in some suitable sense. In practice, the precise definition one adopts depends very much on the applications one has in mind. A first example of higher group is provided by the string group $String(n)$ \cite {bry-mc2, st, schommer:string-finite-dim} in mathematical physics. Historically, this object first arose as the 3-connected cover of $Spin(n)$; one of the possible models for $String(n)$ is given by a crossed module of (infinite-dimensional) Lie groups \cite {baez:str-gp}, which is a type of higher group that will also play a role in the present paper. Another example, which can be regarded as a generalization of the previous one, comes from the integration theory of $L_\infty$-algebras (also known as homotopy Lie algebras) \cite {henriques}. In this case the appropriate definition for the higher group integrating an $L_\infty$-algebra is given in terms of Kan simplicial manifolds (compare Section \ref {sec:universal-cover} below). Yet another example (or definition) originates in connection with Weinstein's quantization program for \mbox {Poisson} manifolds \cite {wx}. This program leads naturally to the notion of \emph {stacky Lie groupoid} \cite {tz, tz2}, and, in particular, to that of \emph {stacky Lie group.} The study of stacky Lie groupoids was the original motivation for the present paper, as we will explain in more detail below.

There is unfortunately neither general agreement on what the standard definition of a stacky Lie group(oid) should be nor on the corresponding terminology; one sometimes refers to stacky Lie groups as {\em Lie 2-groups.} One possibility is to define stacky Lie groups as group objects in the \emph {Hilsum--Skandalis} bicategory $\HiSk$, that is, the bicategory (i.e., weak 2-category) which has Lie groupoids as objects, right principal bibundles (also called \emph {H.S.-morphisms}) as 1-morphisms and smooth biequivariant maps between bibundles as 2-morphisms \cite {blohmann}. This approach has some advantages in that many standard constructions, such as fibred products for instance, can be given a rather explicit description \cite {schommer:string-finite-dim}. However, from a conceptual point of view this is very much like working all the time with a fixed choice of local coordinates (or atlas) when doing differential geometry. In this paper we prefer to adopt a more intrisic foundational framework. Namely, we define stacky Lie groups as group objects in the 2-category $\DSta$ which has differentiable stacks \cite {metzler, bx} as objects, and maps of stacks over the smooth site, resp., isomorphisms between them, as 1-, resp., 2-morphisms (Definition \ref {def: stacky Lie group}).

For the reader's convenience and because of its relevance to the present work, we analyse the relation between these two notions of stacky Lie group in greater detail. Recall that a \textit {Hilsum--Skandalis morphism} (\textit{H.S.-morphism}) from a Lie groupoid $K= \{K_1 \rra K_0\}$ to another one, $K'= \{K'_1 \rra K'_0\}$, is given by a {\em right-principal bibundle}\[
\xymatrix{
K_1 \ar[d] \ar@<-1ex>[d] & E \ar[dl]_{J_l} \ar[dr]^{J_r} & K'_1 \ar[d]
\ar@<-1ex>[d] \\ K_0 & & K'_0,
}
\]that is to say, by a manifold $E$ on which both $K$ and $K'$ act in a compatible way along the \textit{moment maps} $J_l$ and $J_r$ with respective smooth actions\[
\Phi_l : K_1 \times_{\bs, K_0, J_l} E \to E \quad \text{and} \quad \Phi_r: E \times_{J_r, K'_0, \bt} K'_1 \to E,
\]in such a way that the right action of $K'$ on $E$ is \emph{principal:} in other words, $J_l$ is a surjective submersion and $id \times \Phi_r: E \times_{J_r, K'_0, \bt} K'_1 \xto \simeq E \times_{J_l, K_0, J_l} E $ a diffeomorphism. If $E$ is also left-principal (in the obvious sense) then it is called a {\em Morita bibundle,} and it is said to yield a {\em Morita equivalence} between the Lie groupoids $K$ and $K'$. Now, any differentiable stack can be presented by a Lie groupoid, uniquely up to Morita equivalence \cite {bx}; namely, given a differentiable stack $\cX$, a groupoid  presentation $X =\{X_1 \rra X_0\}$ of $\cX$ can be obtained from any representable surjective submersion $X_0 \twoheadrightarrow \cX$ (this is called a {\em chart} or {\em atlas} for $\cX$) by taking the fibred product $X_1:=X_0 \times_{\cX} X_0$. In more rigorous categorical terms, there is a canonical equivalence of bicategories between the 2-category $\DSta$ and the bicategory $\HiSk$ defined above. Consequently, a stacky Lie group can always be presented by a stacky Lie group in the former sense, i.e., by a group object in the bicategory $\HiSk$, very much like a smooth manifold can always be defined by giving a particular atlas for it.

The stacky Lie groupoids which prove to be really important for practical applications are the \emph {\'etale} ones, which are defined to be those whose underlying differentiable stack can be presented by an \'etale Lie groupoid; a generic presentation for such a differentiable stack will be a \emph {foliation} groupoid \cite {moerdijk}. In \cite {z:tgpd-2}, \'etale stacky Lie groupoids are called \emph {Weinstein groupoids.} The isotropy groups of stacky Lie groupoids constitute, for us, the fundamental example of a stacky Lie group. The purpose of this paper is precisely to better understand the structure of these isotropy groups.

Recall that a \emph {crossed module} (of Lie groups) is a pair of Lie groups $(H,G)$ given with a homomorphism $\partial: H \to G$ and a smooth left action $(g,h) \mapsto g*h$ of $G$ on $H$ by automorphisms of $H$ so that the following two axioms are satisfied:
\begin{itemize}
\item[(Eq)] $\partial(g * h) = g \partial(h) g^{-1}$ (Equivariance)
\item[(Pf)] $\partial(h) * h' = h h' h^{-1}$ (Pfeiffer identity).
\end{itemize}
It is our goal here to establish the following result:

\medskip

{\bfseries Theorem \ref{thm: main thm}.} {\em Every connected stacky Lie group $\cG$ is equivalent to a crossed module of the form $(\pi_1(\cG),G)$, where $\pi_1(\cG)$ denotes the fundamental group of $\cG$ (viewed as a discrete Lie group), and $G$ is a connected and simply connected Lie group.}

\medskip

\noindent (For a stacky Lie group, connectedness just means path-connectedness of the underlying differentiable stack, compare Definition \ref {def: connected stacky Lie gp}. The connectedness assumption is natural from the point of view of our applications. The notion of equivalence of stacky Lie groups is defined in Subsection \ref {ssec: group objects in 2-categories}.)

It is always possible to strictify a discrete 2-group; this is a well known result of Baez and Lauda \cite{Baez:2gp}, who provide a proof via group cohomology. However, if a 2-group carries a topology or a smooth structure, then it is not clear how one can achieve the strictification result by the same methods. As an example of these difficulties, the string Lie 2-group (unfortunately it is not one of the \'etale stacky Lie groups we consider) sometimes has a strict but infinite-dimensional model \cite{baez:str-gp}, sometimes has a finite-dimensional but nonstrict model \cite{schommer:string-finite-dim}. Thus, we can see from this example that the strictification procedure is in general highly nontrivial. Morevoer, the strictification method provided in \cite{Baez:2gp} is far from being constructive. By contrast, our method is completely constructive and solves the problem within the \'etale, finite-dimensional world.

Our result is likely to have consequences also for the study of noneffective orbifold groups. Since every orbifold is an \'etale differentiable stack, as soon as the orbifold carries a group structure, an immediate corollary of our result is that one can find a global quotient compatible with the group structure.

We expect the main result of this paper to be relevant also to our program of revisiting the constructions and the results of \cite {tre:2008a, tre:2008b} in terms of stacky Lie groupoids and representations up to homotopy.

\subsubsection*{Acknowledgements}
We thank Rui Fernandes for his invitation to IST, where the project was first exposed.

\section{From stacky Lie groups to semistrict Lie 2-groups}\label{sec: from stacky Lie gps to semistrict Lie 2-gps}

All the material in this section---with the exception of the last statement, Theorem \ref{thm: semi strictification}---is completely standard, and can be found for instance in \cite{Baez:2gp}. We work throughout in a smooth, \'etale context.

\subsection{Background on group objects in 2-categories}\label{ssec: group objects in 2-categories}

Our 2-categories are assumed to always have finite products. In particular, there is always a terminal object $\star$.

\begin{defi}\label{def: group object}
A {\em group object in a 2-category $\mathscr C$} (or {\em $\mathscr C$-group,} for brevity) consists of the following data
\begin{itemize}
\item an object $A \in \mathscr C$
\item a list of 1-morphisms
\begin{itemize}
\item[] \makebox [100pt] [l] {$\mu: A \times A \to A$} (the {\em multiplication})
\item[] \makebox [100pt] [l] {$\eta: \star \to A$} (the {\em unit})
\item[] \makebox [100pt] [l] {$\iota: A \to A$} (the {\em inverse})
\end{itemize}
\item a list of 2-morphisms
\begin{align*}
&a: \mu \circ (\mu \times 1_A) \Rightarrow \mu \circ (1_A \times \mu) & & \quad \text{(the {\em associator})}\\
\begin{split}
&\ell: \mu \circ (\eta \times 1_A) \Rightarrow \pr_A\\
&r: \mu \circ (1_A \times \eta) \Rightarrow \pr_A
\end{split} \Biggr\} & & \quad \text{(the {\em left,} resp.\ {\em right unit constraint})}\\
\begin{split}
&d: \eta \circ \tau_A \Rightarrow \mu \circ (1_A \times \iota) \circ \delta_A\\
&e: \mu \circ (\iota \times 1_A) \circ \delta_A \Rightarrow \eta \circ \tau_A
\end{split} \Biggr\} & & \quad \text{(the {\em adjunction constraints})}
\end{align*}
(where $\tau_A: A \to \star$ denotes the unique 1-morphism from $A$ to the terminal object, and $\delta_A: A \to A \times A$ denotes the diagonal)
\end{itemize}
subject to the requirement that certain coherence conditions hold for which we refer the reader to \cite[p.\ 37]{Baez:2gp}.
\end{defi}

\begin{remark}\label{rmk: monoid object}
The notion of {\em $\mathscr C$-monoid} is obtained from the previous one by neglecting the inversion 1-morphism $\iota$ and the adjunction constraints $d, e$.
\end{remark}

Recall that an {\em \'etale atlas} for a differentiable stack $\mathcal X$ is a representable surjective submersion $X \twoheadrightarrow \mathcal X$ such that the associated Lie groupoid $\{ X \times_{\mathcal X} X \rra X \}$ is \'etale.

\begin{defi}\label{def: DSta}
Let $\DSta$ denote the 2-category whose objects are the differentiable stacks admitting an \'etale atlas, whose 1-morphisms are the maps of differentiable stacks, and whose 2-morphisms are the 2-isomorphisms between maps of differentiable stacks.
\end{defi}

\begin{defi}\label{def: stacky Lie group}
A {\em stacky Lie group} is a group object in the 2-category $\DSta$.
\end{defi}

\begin{defi}\label{def: LGpd}
We denote by $\LGpd$ the 2-category whose objects are the \'etale Lie group\-oids, whose 1-morphisms are the homomorphisms of Lie groupoids (smooth functors), and whose 2-morphisms are the (smooth) natural transformations between homomorphisms of Lie groupoids.
\end{defi}

\begin{defi}\label{def: homs of grp objects}
A {\em homomorphism $\Phi: A \to B$} between two group objects $A, B$ in a 2-category $\mathscr C$ consists of the following data
\begin{itemize}
\item a 1-morphism $\phi: A \to B$
\item a pair of 2-morphisms
\begin{itemize}
\item[] $t: \mu^{(B)} \circ (\phi \times \phi) \Rightarrow \phi \circ \mu^{(A)}$
\item[] $u: \eta^{(B)} \Rightarrow \phi \circ \eta^{(A)}$
\end{itemize}
\end{itemize}
making the appropriate coherence diagrams commute \cite[p.\ 41]{Baez:2gp}.
\end{defi}

\begin{defi}\label{def: transf between homs of grp objects}
A {\em transformation $a: \Phi \rra \Psi$} between two homomorphisms of $\mathscr C$-groups $\Phi, \Psi: A \to B$ is a 2-morphism $a: \phi \to \psi$ compatible with the rest of the structure in the obvious sense \cite[p.\ 42]{Baez:2gp}.
\end{defi}

\begin{defi}\label{def: equiv of grp objects}
An {\em equivalence $A \simeq B$} between two $\mathscr C$-groups $A, B$ is a homomorphism $\Phi: A \to B$ such that there exists a homomorphism $\Psi: B \to A$ together with transformations $\Psi \circ \Phi \rra 1_A$ and $\Phi \circ \Psi \rra 1_B$.
\end{defi}

One has a canonical 2-functor from the 2-category $\LGpd$ into the 2-category $\DSta$; for this well known construction, we refer the reader to \cite{metzler, noohi:top}. Hence

\begin{lemma}\label{lem: gp in LGpd -> stacky Lie gp}
Any $\LGpd$-group $\mathbb G$ canonically determines a corresponding stacky Lie group, and any equivalence of $\LGpd$-groups induces, canonically, one of the corresponding stacky Lie groups. \qed
\end{lemma}

\begin{defi}\label{def: Stacky(bbG)}
We denote by $\Class (\mathbb G)$ the stacky Lie group corresponding to $\mathbb G$ in the above statement.
\end{defi}

\subsection{Coherent Lie 2-groups}\label{ssec: coh Lie 2-gps}

We refer to monoid objects in the 2-category $\LGpd$, in short, $\LGpd$-monoids, also as {\em smooth monoidal groupoids.} For them we adopt the standard notation for monoidal categories; $\TU$ for the unit object; $(x,y) \mapsto x \otimes y$ for the monoidal bifunctor; $a_{x,y,z}: x \otimes (y \otimes z) \to (x \otimes y) \otimes z$ for the associator; $\ell_x: x \otimes \TU \to x$, $r_x: \TU \otimes x \to x$ for the unit constraints.

\begin{defi}\label{def: coh Lie 2-gp}
A {\em coherent Lie 2-group $\mathbb G$} is a smooth monoidal groupoid $\{ G_1 \rra G_0, \otimes, \TU \}$ supplied with the extra structure of
\begin{itemize}
\item a smooth map $\{x \mapsto \overline x\}: G_0 \to G_0$
\item two smooth maps $\{x \mapsto d_x\}, \{x \mapsto e_x\}: G_0 \to G_1$ with
\begin{equation}\label{def: rigidity constraints}
d_x: \TU \to x \otimes \overline x, \qquad e_x: \overline x \otimes x \to \TU
\end{equation}
\end{itemize}
so that the usual {\em adjunction properties} hold \cite[p.\ 10]{Baez:2gp}.
\end{defi}

\begin{defi}\label{def: equiv between coh Lie 2-gps}
A {\em coherent homomorphism $\mathbb G \to \mathbb H$} between coherent Lie 2-groups $\mathbb G, \mathbb H$ consists of
\begin{itemize}
\item[(1)] a homomorphism $\Phi$ between the underlying Lie groupoids
\item[(2)] a monoidal functor structure for $\Phi$, namely, the data of a natural transformation $t_{x, y}: \Phi(x) \otimes' \Phi(y) \to \Phi(x \otimes y)$ between Lie groupoid homomorphisms and of an arrow $u: \TU' \to \Phi(\TU)$ satisfying the standard coherence conditions as in the classical definition of a monoidal functor \cite{maclane:cat-math}.
\end{itemize}

A {\em coherent equivalence} $\mathbb G \xto\thicksim \mathbb H$ is a coherent homomorphism $\mathbb G \to \mathbb H$ which is also a {\em strong equivalence} of the underlying Lie groupoids \cite[Section 5.4]{moerdijk}.
\end{defi}

\begin{lemma}\label{lem: coh Lie 2-gp = gp in LGpd}
A coherent Lie 2-group is exactly the same thing as a $\LGpd$-group, i.e.\ a group object in the 2-category $\LGpd$. To any coherent equivalence between coherent Lie 2-groups, there remains canonically associated an equivalence of $\LGpd$-groups.
\end{lemma}
\begin{proof}
Although the proof is completely standard, we will briefly recall the construction of the inversion homomorphism $i = i_{\mathbb G}: \mathbb G \to \mathbb G$ (compare also the proof of Proposition \ref{prop: strictification of b.-c. semistrict Lie 2-group} below), and make a few additional clarifying remarks.

We set $i_0(x) = \overline x$ on objects. Let $g: x \to y$ be any arrow. Define $\transpose g: \overline y \to \overline x$, the {\em transpose} of $g$, as the following composition
\begin{multline}\label{def: general transpose}
\overline y \xto{ {\ell_{\overline y}}^{-1} } \overline y \otimes \TU \xto{ \overline y \otimes d_x } \overline {y} \otimes (x \otimes \overline x) \xto { \overline y \otimes (g \otimes \overline x) } \overline y \otimes (y \otimes \overline x)
\\
\xto{ a_{\overline y, y, \overline x} } (\overline y \otimes y) \otimes \overline x \xto{ e_y \otimes \overline x } \TU \otimes \overline x \xto{ r_{\overline x} } \overline x.
\end{multline}
One can check that $\transpose {(h \circ g)} = \transpose g \circ \transpose h$; this follows easily from a characterization of the transpose of any arrow $g: x \to y$ as the unique arrow $h: \overline y \to \overline x$ such that the following diagram commutes:
\begin{equation}\label{diagr: characterization of transpose}
\begin{split}
\xymatrix@C=40pt{\overline y \otimes x \ar[d]^{\overline y \otimes g} \ar[r]^-{h \otimes x} & \overline x \otimes x \ar[d]^{e_x} \\ \overline y \otimes y \ar[r]^-{e_y} & \TU.\!\!}
\end{split}
\end{equation}
Then, if we put $i_1(g) = \transpose {(g^{-1})}$ on arrows, we get a functor, and hence a homomorphism of Lie groupoids. Note, conversely, that the above characterization of transposition \eqref{diagr: characterization of transpose} implies that the inversion functor $i: \mathbb G \to \mathbb G$ of any $\LGpd$-group is uniquely determined by the associated data on objects (i.e., by the adjunction data $x \mapsto i_0(x)$, $x \mapsto d_x$, $x \mapsto e_x$). Indeed, by the naturality of $e$ \cite[p.\ 37]{Baez:2gp}, we must have
\begin{align*}
e_x \circ (i_1(g^{-1}) \otimes x) &= e_x \circ [\bigl(i_1(g^{-1}\bigr) \circ \id_{i_0(y)}) \otimes (g^{-1} \circ g)]\\
                                  &= e_x \circ [i_1(g^{-1}) \otimes g^{-1}] \circ [\id_{i_0(y)} \otimes g]\\
                                  &= e_y \circ [i_0(y) \otimes g]
\end{align*}
and therefore, by \eqref{diagr: characterization of transpose}, $i_1(g^{-1}) = \transpose g$.

As to the claim about equivalences, we observe that for each coherent equivalence $\Phi: \mathbb G \xto\thicksim \mathbb H$ between coherent Lie 2-groups one can find a {\em coherent quasi inverse,} namely, a coherent homomorphism $\Psi: \mathbb H \xto\thicksim \mathbb G$ so that there exist {\em monoidal} natural transformations $\Psi \circ \Phi \simeq \id_{\mathbb G}$ and  $\Phi \circ \Psi \simeq \id_{\mathbb H}$. Then, as explained for example in \cite{Baez:2gp}, one can define natural transformations $i_{\mathbb H} \circ \Phi \to \Phi \circ i_{\mathbb G}$ and $i_{\mathbb G} \circ \Psi \to \Psi \circ i_{\mathbb H}$ in such a way as to obtain a pair of $\LGpd$-group homomorphisms forming an equivalence of $\LGpd$-groups.
\end{proof}

\begin{defi}\label{def: semistrict Lie 2-gp}
We call a coherent Lie 2-group $\{ G_1 \rra G_0, \otimes, \TU \}$ {\em semistrict,} if the monoidal bifunctor $\otimes$ makes the manifold of objects $G_0$ into a Lie group with unit $\TU$, and if the constraints (\ref{def: rigidity constraints}) are trivial (that is, $d_x = \id_\TU = e_x$ for all $x \in G_0$).
\end{defi}

In a semistrict (coherent) Lie 2-group, the inverse for each object $x$ is precisely given by $\overline x$.

\begin{defi}\label{def: base connected}
We say that a coherent Lie 2-group is {\em base connected,} when the base manifold of its underlying Lie groupoid is connected.
\end{defi}

\begin{defi}\label{def: connected stacky Lie gp}
We call a stacky Lie group $\cG$ {\em connected,} when for any pair of points $x, y: \star \to \cG$ there exists a path $\R \to \cG$ which restricts to $x$ at zero and to $y$ at one (of course, up to 2-isomorphism).
\end{defi}

We say that a stacky Lie group $\cG$ can be {\em presented} by a coherent Lie 2-group $\mathbb G$, if $\cG$ is equivalent, as a stacky Lie group, to $\Class (\mathbb G)$ (Definition \ref{def: Stacky(bbG)}). Then we claim:

\begin{thm}\label{thm: semi strictification}
Every connected stacky Lie group can be presented by a base connected, semistrict, coherent Lie 2-group.
\end{thm}

The next section will be devoted to proving this theorem.

\section{The universal cover of a stacky Lie group}\label{sec:universal-cover}

Let $\cG$ be an arbitrary connected stacky Lie group, and choose a presentation of its underlying differentiable stack by some Lie groupoid $K_\bullet = \{K_1 \rra K_0\}$. Both $\cG$ and $K_\bullet$ shall be regarded as fixed once and for all throughout the present section.

By Lie II Theorem \cite{z:lie2}, the infinitesimal counterpart of $\cG$ is a Lie algebra $\g$, and the simply connected Lie group $G$ which integrates $\g$ has a canonical projection onto $\cG$
\begin{equation}\label{univ cover map}
p: G \xrightarrow{\;} \cG.
\end{equation}
We are going to establish a few fundamental properties of this map. In order to do this, we first need to review the precise construction of $p$, which involves some technicalities. We shall limit ourselves to the strictly indispensable notions without going into details; the interested reader is referred to \cite [Section 4] {z:lie2} for a complete discussion.

To begin with, we need to introduce yet another point of view on Lie 2-groups, according to which these objects should be defined in terms of simplicial manifolds \cite{henriques, z:tgpd-2}. Even though this approach via simplicial manifolds is very effective, as it allows us to give quick proofs of the results we need, and even though it probably reflects much better the nature of higher groups in general, it has the disadvantage of being not very explicit. For these reasons, and in order not to confuse the reader with too many definitions, no mention of this alternative viewpoint was made within the previous section.

Recall that a \textit{simplicial  manifold $X$} consists of a sequence of {\em manifolds $X_n$}, $n\in \Z^{\ge 0}$ and a collection of {\em smooth} maps (faces and degeneracies) for each $n$
\begin{equation}\label{eq:fd}
\left.
\begin{aligned}
d^n_i: X_n \to X_{n-1} &\text{ (face maps)}\\
s^n_i: X_n \to X_{n+1} &\text{ (degeneracy maps)}
\end{aligned}
\right\}
\qquad \text{for } i \in \{0, 1, 2,\dots, n\}
\end{equation}
satisfying the standard axioms in the definition of a simplicial set (see for example \cite{friedlander}).

\begin{defi}\label{def:defngroupoid} An \textit{$n$-Kan complex $X$}
($n\in\N \cup \infty$) is a simplicial manifold that satisfies the following analogs of the familiar Kan conditions:
\begin{enumerate}
\item for all $m \ge 1$ and $0 \le j \le m$, the restriction map
\begin{equation}\label{Kan restr map}
\mathrm{hom}(\Delta[m],X) \to \mathrm{hom}(\Lambda[m,j],X)
\end{equation}
is a surjective submersion;
\item for each $m > n$ and each $0 \le j \le m$, the same map (\ref{Kan restr map}) is a diffeomorphism.
\end{enumerate}
Here, as usual, $\Delta[m]$ and $\Lambda[m,j]$ denote the fundamental $m$-simplex and its $j$-th horn, respectively. ``$\infty$-Kan complex'' is usually abbreviated into {``Kan complex''.}
\end{defi}

Clearly, a $1$-Kan complex is the same thing as the nerve of a Lie groupoid. This suggests viewing an $n$-Kan complex as the nerve of a Lie $n$-groupoid. (In fact, $n$-Kan complexes are sometimes themselves referred to as ``Lie $n$-groupoids'' in the literature. However, since this usage of the term contrasts with the definitions we adopted in the preceding section, we prefer to stick to the more traditional terminology.) In particular, when $n=2$ and the Kan complex is pointed, namely $X_0=\pt$, we obtain the nerve of an $\HiSk$-group. We briefly recall the explicit correspondence \cite [Section 4] {z:tgpd-2}. Given a Lie groupoid $G_\bullet =\{G_1 \rra G_0\}$ endowed with an $\HiSk$-group structure, the corresponding $2$-Kan complex, which is completely determined by its first three layers and by some structure maps, is given by\[ X_0:=\pt, \quad X_1:=G_1, \quad X_2:=E_m,  \]where $E_m$ is the bibundle presenting the multiplication. We call this associated $2$-Kan complex the {\em nerve of $G_\bullet$}, and we denote it by $NG_\bullet$. The axioms satisfied by the given $\HiSk$-group structure on the groupoid $G_\bullet =\{G_1 \rra G_0\}$ then imply the Kan conditions (Definition \ref {def:defngroupoid}) on the simplicial manifold $NG_\bullet$. Conversely, given a $2$-Kan complex $X$, take $G_0:=X_1$ and $G_1:= d^{-1} _2 (s_0(X_0)) \subset X_2$. Then  $\{d_0, d_1: G_1\rra G_0\}$ is a Lie groupoid, which can be endowed with an $\HiSk$-group structure such that the multiplication bibundle is given by $X_2$.

A {\em local Lie group} is more or less like a Lie group, the difference being in that its multiplication is defined only locally near the identity. More precisely, a local Lie group $G^\mathit{loc}$ is given by two open neighborhoods $V \subset U$ of the origin in $\R^n$, by a smooth multiplication $m: V \times V \to U$, and by a smooth inversion mapping $i: V \to V$, subject to the condition that the usual algebraic axioms should hold whenever they make sense. To any local Lie group $G^\mathit{loc}$ one can still associate a simplicial manifold, the {\em nerve $NG^\mathit{loc}$} of $G^\mathit{loc}$, exactly like one does for groups. However, $NG^\mathit{loc}$ is evidently not a $1$-Kan complex anymore:
\[ (NG^\mathit{loc})_0= \pt, \quad (NG^\mathit{loc})_1=V, \quad (NG^\mathit{loc})_2=m^{-1} (V)\subset V\times V, \]and in general $ (NG^\mathit{loc})_n$ is the following open set of $\R^n$\[(NG^\mathit{loc})_n =\{(g_1, \dotsc, g_n) \in \R^n: g_i\cdot g_{i+1} \dotsm \cdot g_{i+j} \in V, \text{ for all possible $1\le i \le n, 0\le j$} \}. \]
The face and degeneracy maps are exactly like for nerves of groups. Two local Lie groups are {\em isomorphic} if they agree on an open neighborhood of the identity. Local Lie groupoids and their nerves are similarly defined. We refer the reader to \cite[Section 2.1]{z:lie2} for details.

Let us go back to the simply connected Lie group $G$ of (\ref{univ cover map}). We have a local Lie group $G^\mathit{loc}$ defined by any choice of suitably small open sets $V \subset U$ about the identity of $G$ (any two such choices will yield the same result up to isomorphism of local Lie groups). Then, by \cite[Lemma 3.7]{z:tgpd-2}, we can assume that $U$ embeds as an open subset of $K_0$, the manifold of objects of the Lie groupoid $K_\bullet$. Hence we have a Lie groupoid homomorphism, induced by the identity structural embedding $K_0 \to K_1$, from the trivial Lie groupoid $V \rra V$ into $K_\bullet$. This morphism preserves the group-like structure; for example, the multiplication bibundle $E_m$, restricted to $V\times V$, is simply the multiplication map $V \times V \to U$ of $G^\mathit{loc}$ (see \cite[Section 5]{tz} for details). Thus, we obtain a simplicial morphism on the level of nerves \cite[Section 4]{z:lie2} \begin{equation}\label{eq:nerve-local}
NG^\mathit{loc} \to NK_\bullet.
\end{equation}
Then, by applying the operations ``Kan replacement'' ($Kan$) and ``2-truncation'' ($\tau_2$) \cite[Prop.--Def.\ 2.3]{z:kan} to this morphism, we obtain a generalized morphism between $2$-Kan complexes
\begin{equation}\label{cx:gk}
NG \sim \tau_2 \bigl( Kan(NG^\mathit{loc}) \bigr) \to \tau_2 \bigl( Kan(NK_\bullet) \bigr) \sim NK_\bullet ,
\end{equation}
which is a composition of two Morita equivalences (denoted by $\thicksim$) and of a strict morphism. (Here $NG$ is the nerve of the Lie group $G$.) By using the correspondence between $2$-Kan complexes and $\HiSk$-groups, we obtain an $\HiSk$-morphism $G \to K_\bullet$ compatible with the $\HiSk$-group structures. Using the correspondence between differentiable stacks and Lie groupoids mentioned in the introduction, from this $\HiSk$-morphism we finally obtain the desired morphism of stacky Lie groups (\ref {univ cover map}).

A brief digression is perhaps in order at this point to explain where the Morita equivalences in (\ref{cx:gk}) come from. The existence of the first Morita equivalence, $NG \sim \tau_2 \bigl( Kan(G^\mathit{loc}) \bigr)$, is essentially a consequence of the fact that $\pi_2(G)=1$. The details are as follows. To begin with, recall that in general to any Lie algebroid $A$ over a manifold $M$ one can associate a certain infinite-dimensional manifold, $P_aA$, called the {\em $A$-path space} \cite [Section 1] {cf}, and, on $P_aA$, a canonical finite-codimensional foliation, $\cF \equiv \cF(A)$ \cite [Proposition 4.7] {cf}. This foliated manifold determines \cite [Section 5.2] {moerdijk} a corresponding monodromy groupoid $Mon_{\cF}(P_a A)$ over $P_a A$, which represents a certain differentiable stack, $\cG(A)$. There is a canonical stacky groupoid structure over $M$ on $\cG(A)$, which makes the latter into the stacky Lie groupoid integrating $A$ \cite {tz}. As in \cite {z:tgpd-2}, one can form the nerve of the differentiable groupoid $Mon_\cF(P_a A) \rra P_a A$, which will be a $2$-Kan complex. Then

\begin{lemma}\label{lemma:loc-to-algd}
Given a Lie algebroid $A$, let $G^\mathit{loc}(A)$ be its local Lie groupoid. The $2$-Kan complex $\tau_2 \bigl( Kan(NG^\mathit{loc}(A)) \bigr) $ and the nerve of $Mon_\cF(P_a A) \rra P_a A$ are Morita equivalent.
\end{lemma}
\begin{proof}
This follows from the first part of the proof of \cite[Theorem 4.9]{tz} and from \cite[Theorem 3.8 and Proposition 4.1]{z:lie2}.
\end{proof}

\noindent Now, when $A=\g$ is a Lie algebra, one has that $Mon_\cF(P_a \g) \rra P_a \g$ is Morita equivalent to the trivial groupoid $G \rra G$ associated to the Lie group $G$ integrating $\g$ because, in this case, $\pi_2(G)=1$; compare \cite[Remark 5.3]{z:lie2}. This equivalence being also an equivalence of $\HiSk$-groups, it follows that $NG \sim \tau_2 \left( Kan(NG^\mathit{loc}) \right)$ as $2$-Kan complexes. This accounts for the first Morita equivalence appearing in \eqref{cx:gk}. The other Morita equivalence there follows from \cite[Theorem 3.6]{z:kan}, which says that if $X$ is already a $2$-Kan complex then the 2-truncation of the Kan replacement will not change the Morita equivalence class of $X$.

\medskip

Having recalled the necessary technical background about the construction of the map (\ref {univ cover map}), we can now proceed to study its basic properties.

We make an elementary observation:

\begin{lemma}\label{lem: elementary remark}
Given a Lie group $H$ and a smooth map $\varphi: X \to H$, the equation
\[
\Phi(x,y):= \varphi(x) \varphi(y)^{-1}
\]
defines a Lie groupoid homomorphism $\Phi$ from the pair groupoid $P_X = \{X \times X \rra X\}$ into $H$. Conversely, given a Lie groupoid homomorphism $\Phi: P_X \to H$ and a prescribed value $\varphi(x_0) \in H$, one recovers $\varphi$ by setting $\varphi(x) := \Phi(x,x_0) \varphi(x_0)$.

An analogous statement holds for maps of differentiable stacks $\varphi: \cX \to \cH$ into a stacky Lie group $\cH$. \qed
\end{lemma}

These constructions are {\em natural.} For instance, in the stacky case, given another map $\varphi': \cX' \to \cH'$, a map $a: \cX \to \cX'$, and a stacky Lie group homomorphism $\theta: \cH \to \cH'$, commutativity of the first diagram below implies commutativity of the second
$$%
\xymatrix @C=40pt{ \cX \ar[d]^a \ar[r]^-{\varphi} & \cH \ar[d]^\theta & P_\cX \ar[d]^{P_a} \ar[r]^-{\Phi} & \{\cH \rra \pt\} \ar[d]^\theta \\ \cX' \ar[r]^-{\varphi'} & \cH' & P_{\cX'} \ar[r]^-{\Phi'} & \{\cH' \rra \pt\} \text. \!\!}
$$%
The commutativity of the first diagram follows from that of the second one so long as we have $\theta(\varphi(x_0)) \sim \varphi'(a(x_0))$ at some point $x_0: \pt \to \cX$.

\medskip

The first basic property of the map $p$ is {\em surjectivity.} Precisely,

\begin{lemma} \label{lemma:pt-lift}
The map $p: G \to \cG$ is {\em surjective,} in the sense that for any point $x: \pt \to \cG$ one can lift $x$ to some point $\tx: \pt \to G$ making the following diagram 2-commute:
\[
\xymatrix @C=40pt{ & G \ar[d]^p \\ \pt \ar@{.>}[ru]^{\tx} \ar[r]^-{x} & \cG \text. \!\!}
\]
\end{lemma}
\begin{proof}
Let $\eta: \pt \to \cG$ denote the group unit. Since $\cG$ is connected, we can find a map $a: \R \to \cG$ which, up to 2-isomorphism, restricts to $\eta$ at zero and to $x$ at one (recall Definition \ref {def: connected stacky Lie gp}). By Lemma \ref {lem: elementary remark}, $a$ yields a groupoid morphism
\[
P_\R \xrightarrow{a_\mathit{gpd}} \cG
\]
(where $P_\R$ is the pair groupoid $\R \times \R \rra \R$), which differentiates to a Lie algebroid morphism $a_\mathit{algd}: T\R
\to \g$ and then, by Lie II theorem, integrates back to a Lie groupoid morphism
\[
P_\R \xrightarrow{\ta_\mathit{gpd}} G
\]
such that the following diagram of Lie groupoids commutes:
\[
\xymatrix @C=40pt { & G \ar[d]^p \\ P_\R \ar@{.>}[ru]^{\tilde{a}_\mathit{gpd}} \ar[r]^-{a_\mathit{gpd}} & \cG \text.\!\!}
\]
Since $p$ is a stacky group morphism, $\eta$ lifts to $e: \pt \to G$, the identity of $G$. Then, again by Lemma \ref {lem: elementary remark}, we obtain a map $\ta: \R \to G$ lifting $a: \R \to \cG$, since we can choose $\ta(0)=e$ as initial value for $\ta$: in fact, $\ta=\ta_\mathit{gpd}|_{\R \times 0} \cdot e$. Hence $\tx := \ta|_1: \pt \to G$ will lift $x: \pt \to \cG$.
\end{proof}

Our next lemma says that $p$ is, in a sense, a ``Serre fibration''.

\begin{lemma}\label{lemma:gen-lift}
Suppose that the outer square 2-commutes in the diagram below. Then there exists a unique smooth map $\tilde f$ such that both triangles in the diagram 2-commute:
\begin{equation}\label{eq:gen-lift}
\begin{split}
\xymatrix {\R^n \times 0 \ar@{_(->}[d]_{i_0} \ar[r] & G \ar[d]^p \\ \R^n \times \R^k \ar@{.>}[ur]^{\tilde f} \ar[r]^-f & \cG. \!\!}
\end{split}
\end{equation}
(Of course, the upper triangle will then be strictly commutative.)
\end{lemma}
\begin{proof}
Let us put $U := \R^n$ and $V := \R^n \times \R^k$, so that the left vertical map in the diagram reads $i_0: U \hookrightarrow V$.

(Part I. Existence of a lift.)\ \ By Lemma \ref {lem: elementary remark}, we obtain from \eqref {eq:gen-lift} a commutative diagram of stacky Lie groupoids
\begin{equation}\label{diag:pair}
\begin{split}
\xymatrix @C=40pt{
P_U \ar[r] \ar[d]^{P_{i_0}} & G \ar[d]^p \\
P_V \ar[r]^{f_\mathit{gpd}} & \cG, \!\!}
\end{split}
\end{equation}
where $P_\square$ denotes the pair groupoid $\square \times \square \rra \square$. By differentiation of \eqref {diag:pair}, we get the following commutative square of Lie algebroid homomorphisms
\begin{equation}\label{eq:algd-lift}
\begin{split}
\xymatrix @C=40pt{
TU \ar[r] \ar[d] & \g \ar[d]^{id} \\
TV \ar[r]^-{f_\mathit{algd}} \ar@{.>}[ur]^{f_\mathit{algd}} & \g, \!\!}
\end{split}
\end{equation}
in which a lift of $f_\mathit{algd}$ exists uniquely: take $f_\mathit{algd}$ itself. Next, we integrate everything back as we did in the proof of the last lemma. The local exponential map gives a commutative diagram of local Lie groupoids
\begin{equation}\label{eq:local-gpd-lift}
\begin{split}
\xymatrix @C=40pt{
P_U^\mathit{loc} \ar[r] \ar[d] & G^\mathit{loc} \ar[d]^{id} \\
P_V^\mathit{loc} \ar[r] \ar@{.>}[ur] & G^\mathit{loc} .\!\!}
\end{split}
\end{equation}
By applying the nerve functor and by composing in front with the simplicial morphism $NG^\mathit{loc} \to NK_\bullet$ of Eq.~\eqref {eq:nerve-local}, the last diagram is turned into a commutative diagram (of simplicial manifolds) of the form
\[
\xymatrix @C=40pt{
N\bigl(P_U^\mathit{loc}\bigr) \ar[r] \ar[d] & NG^\mathit{loc} \ar[d] \\
N\bigl(P_V^\mathit{loc}\bigr) \ar[r] \ar@{.>}[ur] & NK_\bullet, \!\!}
\]
to which we then apply $\tau_2\bigl(Kan(-)\bigr)$. By Lemma \ref{lemma:loc-to-algd} and the fact that the stacky Lie groupoid $\cG(TU)$ integrating $TU$ equals $P_U$ if $U$ is 2-connected \cite[Section 6]{z:lie2}, we obtain a 2-commutative diagram of stacky Lie groupoids\footnote{For the notion of morphism of stacky Lie groupoids that we are using, see \cite[Section 4]{z:lie2}.}
\begin{equation}\label{eq:pr-gpd-lift}
\begin{split}
\xymatrix @C=40pt{
P_U \ar[r] \ar[d] & G \ar[d]^p \\ P_V \ar@{.>}[ur] \ar[r] & \cG. \!\!}
\end{split}
\end{equation}
Since each one of the morphisms indicated by a solid arrow in this diagram induces the same infinitesimal morphism as its counterpart in the diagram \eqref{diag:pair}, it follows (according to Lie II Theorem, which says that any two morphisms integrating the same infinitesimal morphism can at most differ by a 2-morphism) that the diagonal map in \eqref{eq:pr-gpd-lift} is a lift of $f_\mathit{gpd}$ in \eqref{diag:pair}, which we call $\tf_\mathit{gpd}$. Now to obtain a lift of $f$, we use again the remarks following Lemma \ref {lem: elementary remark}. Namely, choose any point $a_0: \pt \to V$, let $x:= f(a_0,0): \pt \to \cG$, and let $\tx: \pt \to G$ be the point to which $a_0$ is mapped by the upper horizontal arrow in \eqref{eq:gen-lift}, so that in particular $p(\tx) = x$ (compare \ref{lemma:pt-lift}). Choosing precisely $\tx$ as prescribed value at $(a_0,0)$, it follows that $\tf := \tf_\mathit{gpd}|_{V \times a_0} \cdot \tx$ is the required lift of $f$ in \eqref{eq:gen-lift}. The proof of existence is finished.

(Part II. Uniqueness.)\ \ Let two lifts $\tilde{f}$ and $\tf'$ as in \eqref{eq:gen-lift} be given. By the naturality statement in Lemma \ref {lem: elementary remark}, they both give rise to maps lifting $f_\mathit{gpd}$ in \eqref{diag:pair}. These maps must coincide up to a 2-morphism, by Lie II, because at the infinitesimal level the lift is unique. In other words, $\tilde f_\mathit{gpd}$ and $\tilde f'_\mathit{gpd}$ might differ by a 2-morphism; however, since $P_V$ and $G$ happen to be Lie groupoids, the two maps actually coincide. Moreover, since $\tilde{f}$ and $\tf'$ have the same prescribed value at any point $(a_0,0)\in \R^n\times \{0\}$ (by assumption, they are both lifts in \eqref{eq:gen-lift}!), it follows (once again from Lemma \ref {lem: elementary remark}) that $\tilde{f}= \tf'$.
\end{proof}

\begin{cor}
The map $p: G\to \cG$ is a covering map, in the sense that given any map $f$ from $\R^k$ to $\cG$ and any point $g_0: \pt \to G$,
there is a unique map $\tf: \R^k \to G$ such that $\tf(0) = g_0$; more exactly, $\tf$ makes the following diagram 2-commute
\[
\xymatrix @C=40pt{
\pt \ar[r]^-{g_0} \ar[d]_{i_0} & G \ar[d]^p \\
\R^k \ar@{.>}[ur]^{\tf} \ar[r]^-f & \cG. \!\!}
\]
\end{cor}

\begin{cor}\label{cor:epi}
The map $p: G\to \cG$ is an epimorphism of stacks.
\end{cor}
\begin{proof}
Take any object $x: V\to \cG$ of the stack $\cG$. Cover $V$ by contractible open sets $V_i \cong \R^k$. We take $n=0$ and $f$ to be the composition $\R^k \cong V_i \to \cG$ in Lemma \ref {lemma:gen-lift}. By Lemma \ref {lemma:pt-lift}, we can build a diagram of the form \eqref {eq:gen-lift}. Then we have a lift $\tx_i: V_i \to G$, which maps to $x|_{V_i}: V_i \to \cG$ upon composing with $p$. Hence $p$ is an epimorphism.
\end{proof}

Recall that a map $f: \cX \to \cY$ between differentiable stacks is said to be a {\em submersion} if and only if there is a chart $X$ for $\cX$ and a chart $Y$ for $\cY$ such that the map $X \times_{\cY} Y \to Y$ in the diagram below is an ordinary submersion of smooth manifolds. Similarly, $f$ is said to be {\em \'etale} if and only if there are as above \'etale charts $X$ and $Y$ such that the same map $X \times_{\cY} Y \to Y$ is \'etale (a local diffeomorphism). These definitions do not depend on the choice of charts, namely if the condition is satisfied in one pair of charts then it is satisfied in any pair of charts. For any $f$ and any choice of charts $X$ and $Y$, the pullback diagram
\begin{equation}\label{diag: present H.S. bibun}
\begin{split}
\xymatrix{
X\times_{\cY}Y \ar[rr] \ar[d] & & Y \ar@{->>}[d]
\\
X \ar@{->>}[r] & \cX \ar[r]^-f & \cY}
\end{split}
\end{equation}
is in fact an H.S.-morphism, from the Lie groupoid  $X\times_{\cX}X \rra X$ to the Lie groupoid $Y\times_{\cY}Y \rra Y$, presenting the map of differentiable stacks $f:\cX \to \cY$; the manifold $X \times_{\cY} Y$ is the H.S.-bibundle, and the two maps from $X\times_{\cY} Y$ to $X $ and $Y$ are the moment maps. A well known argument in the theory of stacks gives the following two lemmas.

\begin{lemma}\label{lemma:submersion}
A map $\cX \to \cY$ between differentiable stacks is a submersion if and only if for any given groupoid presentations $X = \{X_1 \rra X_0\}$ of $\cX$ and $Y = \{Y_1 \rra Y_0\}$ of $\cY$ the moment map $E \to Y_0$ of any presenting H.S.-bibundle $E$ is a submersion. In particular, if $\cX \to \cY$ can be presented by a strict Lie groupoid homomorphism $X \to Y$ then we can take $E := X_0 \times_{Y_0} Y_1$ and thus $\cX \to \cY$ is a submersion iff  $X_0 \to Y_0$ is a submersion. \qed
\end{lemma}

\begin{lemma}\label{lemma:etale}
A map $\cX \to \cY$ between \'etale differentiable stacks is \'etale if and only if for any given \'etale groupoid presentations $X$ of $\cX$ and $Y$ of $\cY$ as before the moment map $E \to Y_0$ of any presenting H.S.-bibundle $E$ is \'etale. In particular, if $\cX \to \cY$ can be presented by a strict homomorphism $X \to Y$ then we can take $E := X_0 \times_{Y_0} Y_1$ and thus $\cX \to \cY$ is \'etale iff so is  $X_0 \to Y_0$. \qed
\end{lemma}

\begin{lemma}\label{lemma:p-submersion}
The map $p: G\to \cG$ is a submersion.
\end{lemma}
\begin{proof}
Take the groupoid presentations of $G$ and $\cG$ coming from $\tau_2 \bigl( Kan(NG^\mathit{loc}) \bigr)$  and
$\tau_2 \bigl(Kan(NK_\bullet) \bigr)$ respectively. The map $p$ is induced by a strict morphisms of simplicial manifolds $\tau_2 \bigl(Kan(NG^\mathit{loc}) \bigr)\to \tau_2 \bigl(Kan(NK_\bullet) \bigr)$, which is in turn induced by the inclusion $V \to K_\bullet$. Hence $p$ is represented by a strict morphism of groupoids with respect to these presentations. On the level of objects, the map is simply a disjoint union of iterated copies of the inclusion $V \to K_0$:
\[
\tau_2 \bigl( Kan(NG^\mathit{loc}) \bigr)_1 =V \sqcup V\times V \sqcup \dots \lra  \tau_2 \bigl( Kan(NK_\bullet) \bigr)_1=K_0\sqcup K_0 \times K_0 \sqcup \dots
\]
(see \cite[Sec.2]{z:kan} for the complete formula).  Since the inclusion $V\to K_0$ is a submersion, the induced map $\tau_2 \bigl(Kan(NG^\mathit{loc}) \bigr)_1\to \tau_2 \bigl(Kan(NK_\bullet) \bigr)_1$ will be a submersion as well. Hence, by Lemma \ref{lemma:submersion}, the map $p$ must be a submersion.
\end{proof}

\begin{lemma}
The map $p: G \to \cG$ is \'etale.
\end{lemma}
\begin{proof}
In the \'etale groupoid presentations $G$ of $G$ and  $K_\bullet$ of $\cG$, $p$ is represented by an H.S.-bibundle $E_p$. By Lemma \ref{lemma:p-submersion} and Lemma \ref{lemma:submersion}, the moment map $E_p \to K_0$ is a submersion. However, since $E_p$ is a $K_\bullet$-principal bundle over $G$, $E_p$ has the same dimension as $G$ and therefore as $K_0$. Hence the moment map $E_p\to K_0$ is \'etale. By Lemma \ref{lemma:etale}, the map $p$ is itself \'etale.
\end{proof}

\begin{lemma}
The map $p: G\to \cG$ is a representable surjective submersion.
\end{lemma}
\begin{proof}
Since $p$ is an epimorphism by Corollary \ref{cor:epi}, we only need to show that $p$ is a representable submersion. Since $\cG$ is a differentiable stack, we have a chart $\varphi: U \to \cG$. By \cite[Lemma 2.11]{bx}, in order to show that $p$ is a representable submersion it is enough to show that $U \times_{\varphi, \cG, p} G$ is representable and that the map $U \times_{\varphi, \cG, p} G \to U$ below is a submersion:
\[
\xymatrix{
U\times_{\varphi, \cG, p} G \ar[r] \ar[d] & G \ar[d]^{p}\\
U \ar[r]^{\varphi} & \cG. \!\!}
\]
But that $U \times_{\varphi, \cG, p} G$ is representable is clear, because $\varphi$ is  representable, and  $U\times_{\varphi, \cG, p} G \to U$ is a submersion since by Lemma \ref {lemma:p-submersion} $p$ is a submersion.
\end{proof}

Consider the following pullback diagram:
\begin{equation} \label{eq:g1-g}
\begin{split}
\xymatrix{ G_1:=G\times_{\cG} G \ar[r] \ar[d] & G \ar[d]^p \\
G \ar[r]^p &\cG. \!\!}
\end{split}
\end{equation}
Then $\{G_1 \rra G\}$ is an \'etale Lie groupoid presenting the stack $\cG$.

\begin{lemma}\label{lemma:strict-morphism}
Suppose that a morphism $\Phi: \cX \to \cY$ of differentiable stacks and a morphism $\phi: X_0 \to Y_0$ of their charts fit into a 2-commutative diagram of the form
\[
\xymatrix{
X_0 \ar[r]^{\phi} \ar[d]_x & Y_0 \ar[d]^y\\
\cX \ar[r]^{\Phi} \ar@{=>}[ur] & \cY,
}
\]
where $x: X_0 \to \cX$ and $y: Y_0 \to \cY$ are the chart
projection maps. Then the H.S.-morphism presenting $\Phi$ (from the Lie groupoid $X := \{ X_1 = X_0 \times_\cX X_0 \rra X_0\}$ to the Lie groupoid $Y := \{ Y_1 = Y_0 \times_\cY Y_0 \rra Y_0\}$) is strict.
\end{lemma}
\begin{proof} This is proved in detail in \cite[Section 4]{z:lie2}, here we only recall the idea. Consider the following 2-commutative diagram:
\[
\xymatrix{& Y_1 \ar[rr] \ar[dd] & & Y_0 \ar[dd]^y \\
X_1 \ar@{.>}[ur]_{\phi_1} \ar[dd] \ar[rr]& & X_0 \ar[ur]_{\phi}
\ar[dd]^(.4)x \\
& Y_0  \ar[rr]_(.4)y  & & \cY. \!\!  \\
X_0 \ar[ur]_{\phi} \ar[rr]_x & & \cX \ar[ur]_{\Phi}
}
\]
There are two composite maps of the form $X_1 \to X_0 \to Y_0 \to \cY$, the one going through the upper $Y_0$ and the other one going through the lower $Y_0$. These two maps are the same up to a 2-morphism, because the front, back, right, and bottom faces of the diagram  are 2-commutative. Thus, by the universal property of the pullback, there exists a  morphism $\phi_1: X_1 \to Y_1$ making all the faces of the diagram 2-commute. Then $\phi_1$ and $\phi$ together form a groupoid morphism, because $\Phi$ is a morphism of categories fibred in groupoids.
\end{proof}

\begin{cor}\label{cor:m-strict-morphism}
With respect to the groupoid presentation $G_1 \rra G$, the multiplication law $m_\cG$ of $\cG$ can be presented by a strict morphism of Lie groupoids $(m_1,m_G): \{G_1 \times G_1 \rra G\times G\} \to \{G_1 \rra G\}$, where $m_G$ is the multiplication of $G$.
\end{cor}
\begin{proof}
Since the chart projection $p: G \to \cG$ is a morphism of
stacky groups, we have a 2-commutative diagram
\[
\xymatrix{G\times G \ar[r]^-{m_G} \ar[d]_{p\times p} & G \ar[d]^{p}\\
\cG \times \cG \ar[r]^-{m_\cG} \ar@{=>}[ur] & \cG, \!\!}
\]
so the result follows from Lemma \ref{lemma:strict-morphism}.
\end{proof}

Similarly, one can prove that the unit and the inverse of $\cG$
are presentable by strict groupoid morphisms under the groupoid
presentation $G_1 \rra G$. Thus $G_1 \rra G$ is a $\LGpd$-group. Hence, by the correspondence between coherent Lie 2-groups and $\LGpd$-groups mentioned in Lemma \ref{lem: coh Lie 2-gp = gp in LGpd}, we have finally proved Theorem \ref{thm: semi strictification}.

\section{From semistrict Lie 2-groups to crossed modules}\label{sec: from semistrict Lie 2-groups to crossed modules}

In this section we carry out the second step of our strictification procedure. Connectedness and \'etaleness of the objects involved play an essential role in our proofs. We begin by recalling a well known property of monoidal categories.

\begin{lemma}\label{lem: Aut(1) is abelian}
For any semistrict Lie 2-group  $\{ G_1 \rra G_0, \otimes, \TU \}$, the isotropy group $H = t^{-1}(\TU) \cap s^{-1}(\TU)$ is abelian.
\end{lemma}
\begin{proof}
Let $h, h' \in H$. Since $h' = {r_\TU}^{-1} \circ (h' \otimes \TU) \circ r_\TU$ and $h = {\ell_\TU}^{-1} \circ (\TU \otimes h) \circ \ell_\TU = {r_\TU}^{-1} \circ (\TU \otimes h) \circ r_\TU$, the claim $h' \circ h = h \circ h'$ follows at once from the equation \[ (h' \otimes \TU) \circ (\TU \otimes h) = h' \otimes h = (\TU \otimes h) \circ (h' \otimes \TU). \qedhere\]
\end{proof}

\begin{lemma}\label{lem: the associator is trivial}
For any base connected, semistrict Lie 2-group $\{ G_1 \rra G_0, \otimes, \TU \}$, the associator is trivial. In particular, $g_1 \otimes (g_2 \otimes g_3) = (g_1 \otimes g_2) \otimes g_3$ for all arrows $g_1, g_2, g_3 \in G_1$.
\end{lemma}
\begin{proof}
Since $\otimes$ is strictly associative on objects, the associator is an automorphism \[ a_{x, y, z}: x \otimes y \otimes z \to x \otimes y \otimes z \] for all objects $x, y, z \in G_0$. Define a map $h: G_0 \times G_0 \times G_0 \to t^{-1}(\TU) \cap s^{-1}(\TU) = H$ by
\begin{equation}\label{equ: the associator is trivial: no1}
h(x, y, z) = a_{x, y, z} \otimes \overline {x \otimes y \otimes z},
\end{equation}
where $\overline {x \otimes y \otimes z}$ denotes the (identity arrow corresponding to the) inverse of the object $x \otimes y \otimes z$ in the Lie group $(G_0, \otimes, \TU)$. Since $h$ is continuous, since $G_0$ is connected, and since $H$ is discrete, $h$ is a constant map of value $h_0 \in H$.

We contend that $h_0 = \id_\TU$. To begin with, observe that, by \eqref{equ: the associator is trivial: no1} and by Lemma \ref{lem: Aut(1) is abelian},
\begin{equation}\label{equ: the associator is trivial: no2}
h_0 = a_{\TU, \TU, \TU} \otimes \TU = r_\TU \circ a_{\TU, \TU, \TU} \circ {r_\TU}^{-1} = a_{\TU, \TU, \TU} = \ell_\TU \circ a_{\TU, \TU, \TU} \circ {\ell_\TU}^{-1} = \TU \otimes a_{\TU, \TU, \TU}.
\end{equation}
Now, by the pentagon coherence condition for the associator,
$$%
a_{\TU, \TU \otimes \TU, \TU} \circ (a_{\TU, \TU, \TU} \otimes \TU) = (\TU \otimes a_{\TU, \TU, \TU}) \circ a_{\TU, \TU, \TU \otimes \TU} \circ a_{\TU \otimes \TU, \TU, \TU}.
$$%
Hence, by the equations \eqref{equ: the associator is trivial: no2}, \[ {h_0}^2 = {h_0}^3, \] from which our claim follows.

Summarizing, we have shown that $a_{x, y, z} \otimes \overline {x \otimes y \otimes z} = \id_\TU$ for all $x, y, z \in G_0$. We proceed to show that $a_{x, y, z} = \id_{x \otimes y \otimes z}$. Put $u = x \otimes y \otimes z$. Then
\begin{align*}
\id_{x \otimes y \otimes z} &= \id_\TU \otimes u
\\
&= (a_{x, y, z} \otimes \overline {x \otimes y \otimes z}) \otimes u
\\
&= a_{u, \overline u, u} \circ [a_{x, y, z} \otimes (\overline u \otimes u)] \circ {a_{u, \overline u, u}}^{-1}
\\
&= a_{u, \overline u, u} \circ (a_{x, y, z} \otimes \TU) \circ {a_{u, \overline u, u}}^{-1}.
\end{align*}
Thus, $\id_{x \otimes y \otimes z} = a_{x, y, z} \otimes \TU = r_{x \otimes y \otimes z} \circ a_{x, y, z} \circ {r_{x \otimes y \otimes z}}^{-1}$. Hence, $\id_{x \otimes y \otimes z} = a_{x, y, z}$.
\end{proof}

\begin{remark}\label{rmk: smoothness of unit constraints}
Recall that, in view of Definition \ref{def: coh Lie 2-gp}, the unit constraints $\ell_x : x \to \TU \otimes x$ and $r_x : x \to x \otimes \TU$ are given by natural transformations between suitable homomorphisms of Lie groupoids. In particular, there is smooth dependence on the variable $x$.
\end{remark}

\begin{lemma}\label{lem: (1 tensor g) = g}
In any base connected, semistrict Lie 2-group $\{ G_1 \rra G_0, \otimes, \TU \}$, one has $g \otimes \TU = g = \TU \otimes g$ for each arrow $g \in G_1$.
\end{lemma}
\begin{proof}
First we prove the identity in a special case, namely when $g \in \Auto(\TU)$ belongs to the isotropy group at the unit object $\TU$.

So, let $\bs(g) = \bt(g) = \TU$. Then, by the naturality of $\ell$ and the equality of objects $\TU \otimes \TU = \TU$, one gets the identity $\TU \otimes g = \ell_\TU \circ g \circ {\ell_\TU}^{-1}$ in the group $\Auto(\TU)$. Since the latter group is abelian by Lemma \ref{lem: Aut(1) is abelian}, the claim follows.

Next, put $l = \ell_\TU \in \Auto(\TU)$ ($\ell_\TU : \TU \to \TU \otimes \TU = \TU$). For each object $x \in G_0$, we take the composition
$$%
x \xto{\quad \ell_x \quad} \TU \otimes x \xto{\quad l^{-1} \otimes x \quad} \TU \otimes x = x
$$%
and denote it by $\tilde \ell_x \in \Auto(x)$. Since, by Remark \ref{rmk: smoothness of unit constraints}, the map $x \mapsto \ell_x$ is continuous $G_0 \to G_1$, so must be $x \mapsto \tilde \ell_x$. Moreover, since $l \in \Auto(\TU)$,
$$%
\tilde \ell_\TU = (l^{-1} \otimes \TU) \circ \ell_\TU = l^{-1} \circ l = \id_\TU,
$$%
by the already established special case. Now, since $\{ G_1 \rra G_0 \}$ is assumed to be an \'etale Lie groupoid, and since $G_0$ a connected manifold, the map $x \mapsto \tilde \ell_x$ must stay in the identity component $G_0 \subset G_1$ for all $x$, and thus $\tilde \ell_x = \id_x$ for all $x \in G_0$.

Now let $g: x \to x'$ be an arbitrary arrow in $G_1$. By the naturality of $\ell$, the functoriality of $\otimes$, and what we have just observed, the rectangle
$$%
\xymatrix@C=50pt{x \ar[d]^g \ar[r]^-{\ell_x} & \TU \otimes x \ar[d]^{\TU \otimes g} \ar[r]^-{l^{-1} \otimes x} & \TU \otimes x \ar[d]^{\TU \otimes g} \ar@{=}[r] & x \ar[d]^{\TU \otimes g} \\ x' \ar[r]^-{\ell_{x'}} &  \TU \otimes x' \ar[r]^-{l^{-1} \otimes x'} & \TU \otimes x' \ar@{=}[r] & x'}
$$%
commutes, and its long edges are identity arrows. The claim follows.
\end{proof}

\begin{defi}\label{def: strict Lie 2-group}
A {\em strict Lie 2-group} is a group object in the 1-category of (\'etale) Lie groupoids and Lie groupoid homomorphisms.
\end{defi}

\begin{prop}\label{prop: strictification of b.-c. semistrict Lie 2-group}
Let $\mathbb G = \{ G_1 \rra G_0, \otimes, \TU, a_{x, y, z}, \ell_x, r_x \}$ be a base connected, semistrict Lie 2-group. Then the strictification $\Str(\mathbb G) := \{ G_1 \rra G_0, \otimes, \TU \}$, obtained by simply discarding the monoidal constraints $a_{x, y, z}, \ell_x, r_x$, is a strict Lie 2-group, equivalent to $\mathbb G$ as a coherent Lie 2-group.
\end{prop}
\begin{proof}
In view of Lemmas \ref{lem: the associator is trivial} and \ref{lem: (1 tensor g) = g}, in order to show that $\Str(\mathbb G)$ is a strict Lie 2-group, the only thing left to be checked is the existence, for every arrow $g \in G_1$, of an arrow $\overline g$ with $g \otimes \overline g = \id_\TU = \overline g \otimes g$.

Let $g: x \to y$. Define $\transpose{g}: \overline y \to \overline x$, the {\em transpose} of $g$, as
\begin{equation}\label{def: transpose}
\overline y = \overline y \otimes \TU = \overline {y} \otimes x \otimes \overline x \xto {\quad \overline y \otimes g \otimes \overline x \quad} \overline y \otimes y \otimes \overline x = \TU \otimes \overline x = \overline x.
\end{equation}
Then put $\overline g := \transpose {(g^{-1})}: \overline x \to \overline y$ (the {\em contragredient} of $g$). Let us check that $\overline g$ defines an inverse for $g$ in the associative monoid $(G_1, \otimes, \id_\TU)$. We have
\begin{align*}
g \otimes (\overline x \otimes g^{-1})
&= [g \circ \id_x] \otimes [\id_\TU \circ (\overline x \otimes g^{-1})]\\
&= [g \otimes \TU] \circ [x \otimes (\overline x \otimes g^{-1})] && \text{(Exchange Law)}\\
&= g \circ g^{-1} && \text{(Lemmas \ref{lem: (1 tensor g) = g}, \ref{lem: the associator is trivial})}\\
&= \id_y.
\end{align*}
Thus, by Lemma \ref{lem: the associator is trivial},
\begin{align*}
& g \otimes \overline g = g \otimes (\overline x \otimes g^{-1}) \otimes \overline y = \id_y \otimes \overline y = \id_\TU\\
& \overline g \otimes g = \overline x \otimes g^{-1} \otimes (\overline y \otimes (g^{-1})^{-1}) = \overline x \otimes \id_x = \id_\TU.
\end{align*}

Next, we define an equivalence $\Phi$ of coherent Lie 2-groups between $\mathbb G$ and $\mathbb G' = \Str(\mathbb G)$. As the Lie groupoid homomorphism underlying $\Phi$, we simply take the identity endofunctor of the underlying Lie groupoid $\{ G_1 \rra G_0\}$. Thus, $\Phi(x) = x$ and $\Phi(g) = g$ for all $x \in G_0$ and $g \in G_1$. As the tensor functor constraints associated to $\Phi$, namely as
$$%
t_{x, y}: \Phi(x) \otimes' \Phi(y) \to \Phi(x \otimes y) \qquad \text{and} \qquad u: \TU' \to \Phi(\TU),
$$%
we take
\begin{equation}\label{def: tensor functor constraints for G -> Str(G)}
\id_{x \otimes y}: x \otimes y \to x \otimes y \qquad \text{and} \qquad \ell_\TU = r_\TU: \TU = \TU \otimes \TU \to \TU,
\end{equation}
respectively. As shown in the proof of Lemma \ref{lem: (1 tensor g) = g}, $\ell_x = (\ell_\TU \otimes x)$ for all $x \in G_0$. Similarly, $r_x = (x \otimes r_\TU) = (x \otimes \ell_\TU)$ for all $x$. It follows immediately that $\Phi$ is a tensor functor. Since $\Phi$ is also a categorical equivalence (in fact, an isomorphism), the proof is finished.
\end{proof}

Put $\Gamma = \bt^{-1}(\TU)$. This is a closed submanifold of $G_1$, and if $\gamma_1, \gamma_2 \in \Gamma$, then $\gamma_1 \otimes \gamma_2: \bs(\gamma_1) \otimes \bs(\gamma_2) \to \TU \otimes \TU = \TU$ still belongs to $\Gamma$. Hence, by Lemma \ref{lem: (1 tensor g) = g}, $(\Gamma, \otimes)$ is a monoid with unit $1 = \id_\TU$. The preceding proposition immediately implies the following

\begin{cor}\label{cor: Gamma is a group}
$(\Gamma, \otimes, 1)$ is a (discrete Lie) group.
\end{cor}
\begin{proof}
Since $\overline \TU = \TU$, $\gamma \in \Gamma$ implies $\overline \gamma \in \Gamma$.
\end{proof}

\begin{defi}\label{def: the crossed module (Gamma,G)}
Let $\partial: \Gamma \to G_0$ denote the restriction of the source map $\bs: G_1 \to G_0$ to $\Gamma$. Define, for all $\gamma \in \Gamma$ and all $x \in G_0$,
\begin{align}\label{def: action of Gamma on G}
\gamma \cdot x& := \partial(\gamma) \otimes x\\
\label{def: action of G on Gamma}
x * \gamma&     := x \otimes \gamma \otimes \overline x.
\end{align}
\end{defi}

Clearly, $\partial$ is a group homomorphism of $\Gamma = (\Gamma, \otimes, 1)$ into $G_0 = (G_0, \otimes, \TU)$. Moreover, (\ref{def: action of Gamma on G}) is a smooth action of $\Gamma$ on the manifold $G_0$, and (\ref{def: action of G on Gamma}) is a smooth action of the connected Lie group $G_0$ on the discrete manifold $\Gamma$. Hence in fact the latter action must be trivial. (These assertions follow from Lemmas \ref{lem: the associator is trivial} and \ref{lem: (1 tensor g) = g}.)

Let us recall the notion of {crossed module}. A {\em crossed module} consists of a pair of Lie groups $(\Gamma,G_0)$, together with a homomorphism $\partial: \Gamma \to G_0$ and a left action $(x,\gamma) \mapsto x * \gamma$ of $G_0$ on $\Gamma$ by automorphisms of $\Gamma$, such that the following two conditions are satisfied:
\begin{itemize}
\item[(Eq)] $\partial(x * \gamma) = x \partial(\gamma) x^{-1}$ (Equivariance)
\item[(Pf)] $\partial(\gamma) * \gamma' = \gamma \gamma' \gamma^{-1}$ (Pfeiffer identity).
\end{itemize}
To any crossed module one can associate a {strict Lie 2-group,} as follows. The induced left action $(\gamma,x) \mapsto \gamma \cdot x = \partial(\gamma) x$ of $\Gamma$ on $G_0$ defines a translation groupoid $\Gamma \ltimes G_0 = \{ \Gamma \times G_0 \rra G_0 \}$ with source and target given by $\bs(\gamma,x) = x$ and $\bt(\gamma,x) = \gamma \cdot x$ respectively and with composition law given by $(\gamma',x') \circ (\gamma,x) = (\gamma'\gamma,x)$ (whenever $x' = \gamma \cdot x$). At the same time, $G_0$ acts on $\Gamma$, so that the Cartesian product $\Gamma \times G_0$ carries a natural group structure
\begin{equation}\label{def: H wreath G}
(\gamma_1,x_1) \otimes (\gamma_2,x_2) := (\gamma_1 (x_1 * \gamma_2), x_1 x_2).
\end{equation}
Let us denote the resulting Lie group by $\Gamma \rtimes G_0$ (wreath product).

We contend that the structure $(\Gamma, G_0, \partial, *)$ (Definition \ref{def: the crossed module (Gamma,G)}) is a crossed module. To begin with, we note that the following two maps
\begin{align}\label{def: the bijections Psi and Phi}
\Psi: \Gamma \times G_0 \to G_1, &\qquad (\gamma,x) \to \gamma \otimes x\\
\Phi: G_1 \to \Gamma \times G_0, &\qquad g \mapsto \bigl( g \otimes \overline {\bt(g)}, \bt(g) \bigr)
\end{align}
are inverse bijections, by Lemmas \ref{lem: the associator is trivial} and \ref{lem: (1 tensor g) = g}. Furthermore, $\Psi$ is a homomorphism of Lie groupoids from $\Gamma \ltimes G_0$ into $\{ G_1 \rra G_0 \}$, inducing the identity on the base $G_0$. The same map is an isomorphism of groups between $(G_1, \otimes, \id_\TU)$ and the wreath product $\Gamma \rtimes G_0$, as the inverse map $\Phi$ is easily seen to be a group homomorphism. Thus, $(\Gamma \ltimes G_0, \Gamma \rtimes G_0)$ is a strict Lie 2-group, because so is $\Str(\mathbb G)$. In particular, it follows that Definition \ref{def: the crossed module (Gamma,G)} defines a crossed module of Lie groups, the Pfeiffer identity being equivalent to the statement that the composition law of the groupoid $\Gamma \ltimes G_0$ is a group homomorphism with respect to the group structure $\Gamma \rtimes G_0$.

\bigskip

Summarizing, we have proved

\begin{thm}\label{thm: strictification}
Every connected stacky Lie group can be presented by a crossed module $(\Gamma,G_0)$, with $\Gamma$ discrete, and with $G_0$ connected and simply connected.
\end{thm}

\section{Relation to the fundamental group}

Our purpose, in this last section, is to show that there is an isomorphism (of groups) between $\pi_1(\cG)$ (the fundamental group of $\cG$) and the group $\Gamma = (\Gamma, \otimes, 1)$ constructed in the last section (Cor.\ \ref{cor: Gamma is a group}). The isomorphism in question is of course noncanonical, as the construction of $\Gamma$ itself was noncanonical.

\begin{defi} \label{def:pi}
Let $x_0: \pt \to \cX$ be a point of a differentiable stack. The {\em (smooth) $n$-th homotopy set} of $\cX$ at $x_0$, denoted by $\pi_n(\cX;x_0)$, is the set of equivalence classes of maps of differentiable stacks $f: S^n \to \cX$ for which there exists a 2-isomorphism $\alpha$ like in the following diagram
\begin{equation}\label{diag: base point preserving}
\begin{split}
\xymatrix @R=2pt @C=40pt {\pt \ar[dd]_{\text {base pt}} \ar[dr]^{x_0}_(.4){}="b" & \\ & \cX, \\ S^n \ar[ur]_f^(.15){\,}="a"\ar@{=>}@/_.3pc/"a";"b"^-\alpha & }
\end{split}
\end{equation}
modulo the homotopy equivalence relation
$$%
\begin{aligned}
f_0 \thicksim f_1 \qquad \text{if and only if} \qquad &\text{there is} \quad F: S^n\times \R \to
\cX \quad \text{such that}\\
& F(\pt, \cdot) \Leftrightarrow i_0, \quad F(\cdot , 0) \Leftrightarrow f_0, \quad F(\cdot , 1) \Leftrightarrow f_1,
\end{aligned}
$$%
where $i_0: \R \to \cX$ denotes the constant map $\R \to \pt \to \cX$.
\end{defi}

\begin{notn}
For $\cX = \cG$ a stacky group, and $x_0 = \TU: \star \to \cG$ the unit of the stacky group, we shall use the abbreviation $\pi_n(\cG) := \pi_n(\cG;\TU)$.
\end{notn}

The usual group structure on $\pi_n(\cX;x_0)$, $n \geqq 1$ (given by concatenation of loops for $n=1$) makes still sense in view of the following

\begin{lemma}
Any element $[g] \in \pi_n(\cX;x_0)$ has a representative $g: S^n \to \cX$ which is constant near the base point $x_0$; namely, there exists an open subset $x_0 \in U \subset S^n$ such that the restriction of $g$ to $U$ factors through $x_0: \star \to \cX$ (up to 2-isomorphism). \qed
\end{lemma}

Obviously, any equivalence of stacky Lie groups $\cG \stackrel \simeq
\to \cG'$ canonically induces isomorphisms of groups $\pi_n(\cG)
\stackrel \thicksim \to \pi_n(\cG')$ for all $n \geqq 1$. Hence, for
our puposes, it will be no loss of generality to assume that $\cG$
actually {\em is} a crossed module of the type considered in the
previous section. Then we have the following immediate consequence
(see also the proof of Lemma \ref{lem: choice of lift doesn't matter}):

\begin{lemma}\label{lem: p is a smooth fibration}
The stack map $p: G \to \cG$ is a representable surjective submersion, with the following lifting property: given a 2-commutative square
\begin{equation}\label{diag: stack fibration lifting property}
\begin{split}
\xymatrix @R=15pt @M=5pt {M \ar@{^(->}[d]_{i_0} \ar[r] & G \ar[d]^p \\ M \times \R \ar[r] \ar@{-->}[ur] & \cG}
\end{split}
\end{equation}
with $M$ a manifold, there exists a unique map of differentiable stacks, as indicated in the diagram, for which the upper and lower triangle 2-commute (the upper triangle, of course, will be then strictly commutative). \qed
\end{lemma}

We say that $p: G \to \cG$ is a {\em smooth fibration.}

\begin{lemma}\label{lem: existence of special liftings}
For any loop $\ell$ representing a class $[\ell] \in \pi_1(\cG)$, there exists some smooth lift $\lambda$ fitting as indicated in the following diagram
\begin{equation}\label{diag: special lift}
\begin{split}
\xymatrix @C=40pt @M=5pt {\star \ar@{^(->}[d]_{i_0} \ar@/^2.1pc/[rr]_-{\TU} & S^1 \times_\cG G \ar[d]_-a^(.45){\;}="a" \ar[r]^-b_(.5){}="b" \ar@/_.5pc/@{=>}"a";"b"_{\theta_{\ell,p}} & G \ar[d]^p \\ \R \ar@{-->}[ur]^-\lambda \ar[r]^-{\exp 2\pi i} & S^1 \ar[r]^-\ell & \cG; \!\!}
\end{split}
\end{equation}
in other words, $\lambda$ is a lift of $t \mapsto \exp(2\pi it)$ to $S^1 \times_\cG G$ with the property that $(b \circ \lambda)(0) = \TU \in G$, or, equivalently, that $b\circ \lambda$ is a lift of $\ell \circ [\exp 2\pi i]$ through $\TU \in G$.
\end{lemma}
\begin{proof}
We have a 2-isomorphism between the two stack morphisms\[ \ell \circ [\exp 2\pi i] \circ i_0 \qquad \text{and} \qquad p\circ \TU, \]by (\ref{diag: base point preserving}), because of the assumption $\ell \in \pi_1(\cG)$. Hence, by Lemma \ref{lem: p is a smooth fibration}, there exists a unique lift $f: \R \to G$ with $f \circ i_0 = \TU$ (i.e., $f(0) = \TU$) and with $p\circ f$ 2-isomorphic to $\ell \circ [\exp 2\pi i]$. By the pullback property of $S^1 \times_\cG G$, we find $\lambda: \R \to S^1 \times_\cG G$ such that $b\circ \lambda = f$ and $a\circ \lambda = \exp 2\pi i$. This is precisely what we wanted.
\end{proof}

\begin{remark}
Of course, the preceding Lemma holds for any choice of an HS-bibundle $E$ representing $\ell$, not just for the canonical pullback of stacks $E = S^1 \times_\cG G$. A similar remark applies to the next result.
\end{remark}

To correctly understand the next lemma, recall that there is a canonical HS-bibundle structure on\[ S^1 \xleftarrow{\quad a \quad} E := S^1 \times_\cG G \xrightarrow{\quad b \quad} G, \]and, therefore, a canonical principal right action of the Lie groupoid $\cG = \{G_1 \rra G\}$ on $E$ along the map $b$. Hence, for any pair of elements $e_0, e_1 \in E$ with $a(e_0) = a(e_1)$, we have a unique arrow $g: b(e_1) \to b(e_0) \in G_1$, denoted by ${e_0}^{-1}e_1$, such that $e_1 = e_0 \cdot g$.

\begin{lemma}\label{lem: choice of lift doesn't matter}
The difference $\lambda(0)^{-1}\lambda(1)$ is the same for all the liftings $\lambda: \R \to S^1 \times_\cG G$ (associated with a given representative loop $\ell$, fixed once and for all) that were considered in the previous lemma, i.e., those liftings fitting in the diagram (\ref{diag: special lift}).
\end{lemma}

The proof will make use of the following simple observation:

\begin{lemma}\label{lem: stabilizer at 1 is in the center of Gamma}
The stabilizer subgroup $\Aut(\TU)= \bs^{-1}(\TU) \cap \bt^{-1}(\TU)$ is contained in the center of the group $\Gamma = (\Gamma, \otimes, 1)$.
\end{lemma}
\begin{proof}
Let $\gamma \in \Gamma$, and $\alpha \in \Aut(\TU)$. Then
\begin{align*}
\gamma \otimes \alpha
&= (\id_\TU \circ \gamma) \otimes (\alpha \circ \id_\TU)\\
&= (\id_\TU \otimes \alpha) \circ (\gamma \otimes \id_\TU) && \text{(exchange law)}\\
&= \alpha \circ \gamma && \text{(Lemma \ref{lem: (1 tensor g) = g})}\\
&= (\alpha \otimes \id_\TU) \circ (\id_\TU \otimes \gamma) && \text{(Lemma \ref{lem: (1 tensor g) = g})} \\
&= (\alpha \circ \id_\TU) \otimes (\id_\TU \circ \gamma) && \text{(exchange law)} \\
&= \alpha \otimes \gamma. \qedhere
\end{align*}
\end{proof}

\begin{proof}[Proof of Lemma \ref{lem: choice of lift doesn't matter}]
Since $\cG$ is presented by the action groupoid $\Gamma \ltimes G \rra G$, an HS-bibundle of the kind considered above is actually the same thing as a principal right $\Gamma$-bundle $a: E \to S^1$ given with a $\Gamma$-equivariant map $b: E \to G$, where $\Gamma$ acts on the right on $G$ by $x \cdot \gamma := \gamma^{-1} \cdot x$. Hence, if $\lambda, \mu$ are two liftings of the kind considered in (\ref{diag: special lift}), and they differ at zero by an element $\gamma_0$, namely $\mu(0) = \lambda(0) \cdot \gamma_0$, they will differ by the same $\gamma_0$ for all $t$, because of the uniqueness of lifting for a given initial condition (the map $a: E \to S^1$ is \'etale, because of the discreteness of $\Gamma$). Thus, there exists ${\alpha_0} \in \Gamma$ such that
\begin{align*}
\mu(0) &= \lambda(0) \cdot {\alpha_0} \qquad \text{and}\\
\mu(1) &= \lambda(1) \cdot {\alpha_0}.
\end{align*}
One necessarily has ${\alpha_0} \in \Aut(\TU)$, because $b(\lambda(0)) = \TU = b(\mu(0))$.

Now, from the assumption
\begin{align*}
\lambda(1) &= \lambda(0) \cdot \gamma \qquad \text{and}\\
\mu(1) &= \mu(0) \cdot \delta,
\end{align*}
it follows, by the principality of the action of $\Gamma$ on $E$, that\[ \delta = {\alpha_0}^{-1} \gamma {\alpha_0} \quad \text{in } \Gamma.\]By Lemma \ref{lem: stabilizer at 1 is in the center of Gamma}, we conclude that $\gamma=\delta$, as contended.
\end{proof}

As observed in the course of the last proof, the assumption $b(\lambda(0)) = \TU$ implies that the difference $\lambda(0)^{-1} \lambda(1)$ is an element of $\bt^{-1}(\TU) = \Gamma$. Thus, we obtain a well defined map into $\Gamma$ from the set of representative loops; to each representative loop $\ell$, one associates the boundary difference $\partial_1(\ell) := \lambda(0)^{-1} \lambda(1)$, for an arbitrary lifting $\lambda$ as in Lemma \ref{lem: existence of special liftings}.

\begin{lemma}\label{lem: difference map descends to homotopy classes}
If two loops $\ell, \ell': S^1 \to \cG$ represent the same homotopy class $[\ell] = [\ell'] \in \pi_1(\cG)$, then their boundary differences are equal: $\partial_1(\ell) = \partial_1(\ell') \in \Gamma$.
\end{lemma}
\begin{proof}
Suppose, as a first step, that there is a 2-isomorphism $\alpha$ relating $\ell$ and $\ell'$
$$%
\xymatrix @C=15pt {S^1 \ar@/^.8pc/[rr]^-\ell \ar@/_.8pc/[rr]_-{\ell'} & \Uparrow & \cG.}
$$%
Let $S^1 \smash[t]{\xleftarrow {\quad a\quad}} E := S^1 \times_\cG G \smash[t]{\xrightarrow {\quad b\quad}} G$ be the pullback of $p$ along $\ell$, and let $E', a', b'$ be the analogous pullback along $\ell'$. Recall that both $a, a'$ are principal right $\Gamma$-bundles (canonically), and that both $b, b'$ are equivariant maps.

By the stacky pullback universal property, there exists a canonical smooth map $\tilde{\alpha}: E' \to E$, which is $\Gamma$-equivariant, and which commutes with the HS-bibundle maps: $a \circ \tilde{\alpha} = a'$, and $b \circ \tilde{\alpha} = b'$.

Now, choose any lifting $\lambda': \R \to E'$, with $a' \circ \lambda' =  \exp(2\pi i\, \text-)$, and with $(b' \circ \lambda')(0) = \TU$. The composition $\lambda := \tilde{\alpha} \circ \lambda'$ then satisfies $a \circ \lambda =  \exp(2\pi i\, \text-)$, $(b \circ \lambda)(0) = \TU$, and is therefore itself a lifting of the type considered in (\ref{diag: special lift}). By the $\Gamma$-equivariance of $\tilde{\alpha}$, the boundary differences for $\lambda'$ and $\lambda$ must be the same. This proves the lemma in the special case $\exists\alpha: \ell' \rra \ell$.

Next, let $L: S^1 \times \R \to \cG$ be a homotopy between the loops $\ell_0 := L(\text-,0)$ and $\ell_1 := L(\text-,1)$. We want to show that $\partial_1(\ell_0) = \partial_1(\ell_1)$. By the same argument used in the proof of Lemma \ref{lem: existence of special liftings}, we can find a lifting $\Lambda: \R \times \R \to (S^1 \times \R) \times_\cG G$ such that $\pr_1 \circ \Lambda = \exp (2\pi i\, \text-) \times \id_\R$ and $\pr_2(\Lambda(0,s)) = \TU \in G$ $\forall s \in \R$, where $\pr_1, \pr_2$ denote the two projections\[ S^1 \times \R \longleftarrow (S^1 \times \R) \times_\cG G \longrightarrow G. \]Put $\ell_s := L(\text-,s)$, for each $s \in \R$. One has a canonical identification between the fiber ${\pr_1}^{-1}(S^1 \times \{s\})$ and the pullback $S^1 \times_\cG G$ along the loop $\ell_s$. For each $s \in \R$, $\lambda_s := \Lambda(\text-,s)$ gets then identified to a lifting of the type considered in (\ref{diag: special lift}) relative to $\ell_s$. Then, the map $s \mapsto \lambda_s(0)^{-1} \lambda_s(1)$ yields a smooth path in $\Gamma$ connecting $\partial_1(\ell_0)$ and $\partial_1(\ell_1)$.
\end{proof}

The last lemma shows that there is a well defined {\em boundary map}
\begin{equation}\label{equ: boundary map}
\partial_1: \pi_1(\cG) \longrightarrow \Gamma.
\end{equation}
This map is the precise analogue of the usual boundary map in the long exact sequence of homotopy groups associated with the ``stack fibration'' $\Gamma \hookrightarrow G \to \cG$.

\begin{lemma}\label{lem: surjectivity boundary map}
The boundary map \eqref{equ: boundary map} is a surjection.
\end{lemma}
\begin{proof}
Let $\gamma_0 \in \Gamma$ be given. We will construct a loop $\ell_0: S^1 \to \cG$ with $\partial_1(\ell_0) = \gamma_0$. The construction will of course make use of the connectedness of the base $G$.

Suppose we have constructed a smooth curve $f: \R \to G$ with the properties $f(0) = \TU$ and $f(t) = \gamma_0 \cdot f(t+1)$ $\forall t \in \R$. Then $\ell_0$ may be obtained as follows. Put
\begin{equation}\label{def: E = (R times Gamma)/ equivalence}
E := (\R \times \Gamma)/\thicksim, \qquad \text{where} \qquad (t,\gamma) \thicksim (t+k,{\gamma_0}^{-k}\gamma) \quad \forall k \in \Z.
\end{equation}
This is evidently a smooth manifold. Define two projections\[ S^1 \xleftarrow{\quad a\quad} E \xrightarrow{\quad b\quad} G \]by setting
\begin{align*}
a([t,\gamma]) &:= \exp(2\pi it)\\
b([t,\gamma]) &:= \gamma^{-1} \cdot f(t),
\end{align*}
where $[t,\gamma]$ denotes the equivalence class of the pair $(t,\gamma)$ with respect to the equivalence relation \eqref{def: E = (R times Gamma)/ equivalence}. Finally, let $\Gamma$ act on $E$ from the right by\[ [t,\gamma] \cdot \gamma' := [t,\gamma\gamma']. \]One obtains in this way an HS-bibundle representing a loop $\ell_0$. By choosing the lifting $\lambda = \{t \mapsto [t,1]\}: \R \to E$, one immediately sees that $\partial_1(\ell_0) = \gamma_0$.

There only remains to show how to construct a curve $f$ with the desired properties. Choose first any smooth curve $\alpha: (-\tfrac14,\tfrac14) \to G$ with $\alpha(0) = \TU$, and translate it by ${\gamma_0}^{-1}$, namely, consider $\beta: (\tfrac34,\tfrac54) \to G$ given by $\beta(t) ={\gamma_0}^{-1} \cdot \alpha(t-1)$. By connectedness of $G$, one can then find a smooth path $f: (-\tfrac18,\tfrac98) \to G$ such that $f$ restricts to $\alpha$ on $(-\tfrac18,\tfrac18)$, and to $\beta$ on $(\tfrac78,\tfrac98)$. Finally, one extends $f$ to all of $\R$ simply by imposing the required $\gamma_0$-periodicity $f(t) = \gamma_0 \cdot f(t+1)$.
\end{proof}

\begin{lemma}\label{lem: injectivity boundary map}
The boundary map \eqref{equ: boundary map} is an injection.
\end{lemma}
\begin{proof}
Let $\ell, \ell': S^1 \to G$ be two loops such that $\partial_1(\ell) = \partial_1(\ell') =: \gamma_0$, and let\[ S^1 \xleftarrow{\quad a\quad} E \xrightarrow{\quad b\quad} G, \qquad S^1 \xleftarrow{\quad a'\quad} E' \xrightarrow{\quad b'\quad} G \]be the corresponding HS-bibundles. Choose respective liftings $\lambda: \R \to E$, $\lambda': \R \to E'$ of the type considered in Lemma \ref{lem: existence of special liftings}, and put $f := b \circ \lambda$, $f' := b' \circ \lambda'$.

Note that one has the periodicity relations $f(t) = \gamma_0 \cdot f(t+1)$, $f'(t) = \gamma_0 \cdot f'(t+1)$, by the assumption $\partial_1(\ell) = \partial_1(\ell') = \gamma_0$. Then, by using the same technique as in the previous proof, one can construct a homotopy $F: \R \times \R \to G$ between $f = F(\text-,0)$ and $f' = F(\text-,1)$ with the periodicity property $F(t,s) = \gamma_0 \cdot F(t+1,s)$ $\forall s,t \in \R$. (This uses the simply connectedness of $G$.)

One obtains again a bibundle
\begin{align*}
&E := (\R \times \R \times \Gamma)/\thicksim, \qquad \text{with} \quad (t,s,\gamma) \thicksim (t+k,s,{\gamma_0}^{-k}\gamma) \quad \forall k \in \Z\\
&E \ni [t,s,\gamma] \mapsto (\exp(2\pi it),s) \in S^1 \times \R\\
&E \ni [t,s,\gamma] \mapsto \gamma^{-1} \cdot F(t,s) \in G\\
&[t,s,\gamma] \cdot \gamma' := [t,s,\gamma\gamma'].
\end{align*}
The reader can check that this gives a homotopy between $\ell$ and $\ell'$.
\end{proof}

\begin{lemma}\label{lem: the boundary map is a group homomorphism}
The boundary map \eqref{equ: boundary map} is a homomorphism of groups.
\end{lemma}
\begin{proof}
Let $\ell_0, \ell_1$ be any two loops in $\cG$. We must show that $\partial_1(\ell_0 \odot \ell_1) = \partial_1(\ell_0) \partial_1(\ell_1)$, where $\ell_0 \odot \ell_1$ denotes the concatenation of the two loops.

The idea behind the proof is very simple. One considers the following map of differentiable stacks\[ S^1 \times S^1 \xto {\quad \ell_0 \times \ell_1 \quad} \cG \times \cG \xto {\quad \mu \quad} \cG\]and its composition $\ell_0 * \ell_1$ with the diagonal embedding $S^1 \hookrightarrow S^1 \times S^1$. By considering HS-bibundle presentations for $\ell_0, \ell_1$ and then for the composition $\mu \circ (\ell_0 \times \ell_1)$, and by playing a bit with liftings, one can explicitly check that\[ \partial_1(\ell_0 * \ell_1) = \partial_1(\ell_0) \partial_1(\ell_1). \]Moreover, by composing the above map of differentiable stacks with the exponential covering $\R \times \R \to S^1 \times S^1$, one sees that the loops $\ell_0 \odot \ell_1$ and $\ell_0 * \ell_1$ correspond to the boundary of the triangle above the diagonal in the square $[0,1] \times [0,1] \subset \R \times \R$.
$$%
\xymatrix @C=45pt @R=45pt {\ar[r]|{\ell_1} \ar@{}[]|{\odot} & \ar@{}[]|{\star} \\ \ar@{}[]|{\star} \ar[u]|{\: \ell_0 \:} \ar@{-}[r] \ar[ur]|{\ell_0 * \ell_1} & \ar@{-}[u]}
$$%
Modulo some obvious technicalities, this shows that $\ell_0 \odot \ell_1$ is homotopic to $\ell_0 * \ell_1$. The proof is now complete.
\end{proof}

We may summarize our conclusions as follows:

\begin{thm}\label{thm: main thm}
Every connected stacky Lie group $\cG$ can be presented as a crossed module of the form $(\pi_1(\cG),G)$, with $\pi_1(\cG)$ the fundamental group of $\cG$ (viewed as a discrete Lie group), and $G$ a connected and simply connected Lie group. \qed
\end{thm}

{\small
\bibliographystyle{../bib/habbrv}
\bibliography{../bib/bibz}

\def\cprime{$'$} \def\cprime{$'$} \def\cprime{$'$} \def\cprime{$'$}
\begin{thebibliography}{10}

\bibitem{Baez:2gp}
J.~C. Baez and A.~D. Lauda.
\newblock Higher-dimensional algebra. {V}. 2-groups.
\newblock {\em Theory Appl. Categ.}, 12:423--491 (electronic), 2004.

\bibitem{baez:str-gp}
J.~C. Baez, D.~Stevenson, A.~S. Crans, and U.~Schreiber.
\newblock From loop groups to 2-groups.
\newblock {\em Homology, Homotopy Appl.}, 9(2):101--135, 2007.

\bibitem{bx}
K.~Behrend and P.~Xu.
\newblock {Differentiable Stacks and Gerbes}, arxiv:math.DG/0605694.

\bibitem{blohmann}
C.~Blohmann.
\newblock Stacky {L}ie groups.
\newblock {\em Int. Math. Res. Not. IMRN}, pages Art. ID rnn 082, 51, 2008.

\bibitem{bry-mc2}
J.-L. Brylinski and D.~A. McLaughlin.
\newblock The geometry of degree-{$4$} characteristic classes and of line
  bundles on loop spaces. {II}.
\newblock {\em Duke Math. J.}, 83(1):105--139, 1996.

\bibitem{cf}
M.~Crainic and R.~L. Fernandes.
\newblock Integrability of {L}ie brackets.
\newblock {\em Ann. of Math. (2)}, 157(2):575--620, 2003.

\bibitem{friedlander}
E.~M. Friedlander.
\newblock {\em \'{E}tale homotopy of simplicial schemes}, volume 104 of {\em
  Annals of Mathematics Studies}.
\newblock Princeton University Press, Princeton, N.J., 1982.

\bibitem{henriques}
A.~Henriques.
\newblock Integrating {$L\sb \infty$}-algebras.
\newblock {\em Compos. Math.}, 144(4):1017--1045, 2008.

\bibitem{maclane:cat-math}
S.~MacLane.
\newblock {\em Categories for the working mathematician}.
\newblock Springer-Verlag, New York, 1971.
\newblock Graduate Texts in Mathematics, Vol. 5.

\bibitem{metzler}
D.~Metzler.
\newblock {Topological and smooth stacks}, arxiv:math.DG/0306176.

\bibitem{moerdijk}
I.~Moerdijk and J.~Mr{\v{c}}un.
\newblock {\em Introduction to foliations and {L}ie groupoids}, volume~91 of
  {\em Cambridge Studies in Advanced Mathematics}.
\newblock Cambridge University Press, Cambridge, 2003.

\bibitem{noohi:top}
B.~Noohi.
\newblock {Foundations of Topological Stacks I}, arXiv:math.AG/0503247.

\bibitem{schommer:string-finite-dim}
C.~Schommer-Pries.
\newblock {A Finite Dimensional String 2-Group}, arXiv:0911.2483v1.

\bibitem{st}
S.~Stolz and P.~Teichner.
\newblock What is an elliptic object?
\newblock In {\em Topology, geometry and quantum field theory}, volume 308 of
  {\em London Math. Soc. Lecture Note Ser.}, pages 247--343. Cambridge Univ.
  Press, Cambridge, 2004.

\bibitem{tre:2008b}
G.~Trentinaglia.
\newblock On the role of effective representations of {L}ie groupoids.
\newblock {\textit {Advances in Mathematics} (2010),} DOI:
  10.1016/j.aim.2010.03.014.

\bibitem{tre:2008a}
G.~Trentinaglia.
\newblock {T}annaka duality for proper {L}ie groupoids.
\newblock {\em J. Pure Appl. Algebra}, 214(6):750--768, 2010.

\bibitem{tz}
H.-H. Tseng and C.~Zhu.
\newblock Integrating {L}ie algebroids via stacks.
\newblock {\em Compos. Math.}, 142(1):251--270, 2006.

\bibitem{tz2}
H.-H. Tseng and C.~Zhu.
\newblock Integrating {P}oisson manifolds via stacks.
\newblock {\em Travaux math\'ematique}, 15:285--297, 2006.

\bibitem{wx}
A.~Weinstein and P.~Xu.
\newblock Extensions of symplectic groupoids and quantization.
\newblock {\em J. Reine Angew. Math.}, 417:159--189, 1991.

\bibitem{z:lie2}
C.~Zhu.
\newblock {Lie II theorem for Lie algebroids via higher groupoids},
  arxiv:math/0701024v2 [math.DG].

\bibitem{z:kan}
C.~Zhu.
\newblock Kan replacement of simplicial manifolds.
\newblock Letters in Mathematical Physics: Volume 90, Issue 1 (2009), Page 383,
  2008, arXiv:0812.4150.

\bibitem{z:tgpd-2}
C.~Zhu.
\newblock {$n$}-groupoids and stacky groupoids.
\newblock {\em Int. Math. Res. Not. IMRN}, 2009(21):4087--4141, 2009.

\end{thebibliography}
}

\end{document}